\documentclass[conference]{IEEEtran}
\IEEEoverridecommandlockouts

\usepackage{cite}
\usepackage{amsmath,amssymb,amsfonts,amsthm}
\usepackage{algorithmic}
\usepackage{graphicx}
\usepackage{textcomp}
\usepackage{xcolor}
\usepackage{multirow}
\usepackage{enumitem}
\usepackage{float}
\usepackage[colorlinks=true, linkcolor=blue, citecolor=blue, urlcolor=blue]{hyperref}
\newtheorem{theorem}{Theorem}
\newtheorem{lemma}[theorem]{Lemma}

\def\BibTeX{{\rm B\kern-.05em{\sc i\kern-.025em b}\kern-.08em
    T\kern-.1667em\lower.7ex\hbox{E}\kern-.125emX}}
\begin{document}
\title{List Reconstruction Problem with List Size Two}

\author{\textbf{Binh Vu}\IEEEauthorrefmark{1},
        \textbf{Shuche Wang}\IEEEauthorrefmark{2},
        and \textbf{Van Khu Vu}\IEEEauthorrefmark{1}\\[0.5mm]
\IEEEauthorblockA{\IEEEauthorrefmark{1} \small VinUniversity, Hanoi, Vietnam \\[1mm]}
\IEEEauthorblockA{\IEEEauthorrefmark{2} \small National University of Singapore, Singapore\\[1mm]}
{\footnotesize binh.vn@vinuni.edu.vn, shuche.wang@u.nus.edu, khu.vv@vinuni.edu.vn  }
}

\maketitle
\begin{abstract}
The problem of computing the cardinality of the intersection of multiple balls in the Hamming space has attracted a lot of attention recently due to their applications in the list reconstruction problem and information retrieval in Associative Memories.
In previous work, most of the results are for the cases where the radii of each ball, $r$ and the distance between the centers of these balls, $k$ are fixed when the length $n$ of each codeword tend to infinity. In this work, we focus on the case where $r = \alpha n$ and $k=\beta n$ for some constants $\alpha$ and $\beta$ and compute the maximum asymptotic rate of the cardinality of the intersection of three balls. We provide the maximum asymptotic rate as a function of two parameters $\alpha$ and $\beta$. We also provide numerical results and compare these results with the intersection of two balls.
\end{abstract}

\section{Introduction}
Let $\Sigma^n=\{0,1,\ldots q-1\}^n$ denote the $q$-ary Hamming space, and let
$d(\cdot,\cdot)$ be the Hamming distance. For each $i \in \Sigma$, let $i^n$ denote the element whose coordinates are all equal to $i$. For $x\in\Sigma^n$, the Hamming ball
of radius $r$ centered at $x$ is $B_r(x)=\{y\in\Sigma^n:d(x,y)\le r\}$.
Given $m$ centers $x_1,\ldots,x_m\in\Sigma^n$ with pairwise distances at least
$k$, we consider the maximum possible size of the common intersection
\begin{equation*}
N_r^{(n)}(q,m,k)
=
\max_{\substack{x_1,\ldots,x_m\in\Sigma^n\\ d(x_i,x_j)\ge k,\ i\ne j}}
\left|
\bigcap_{i=1}^m B_r(x_i)
\right|.
\end{equation*}
This quantity measures the largest possible ambiguity remaining after $m$
distance constraints are imposed.

When $q=2$, we simply write $N_r^{(n)}(2,m,k) = N_r^{(n)}(m,k)$. The study of such intersections is closely connected to Levenshtein's
sequence reconstruction problem, where a transmitted word is to be recovered
from several noisy outputs. This line of work was initiated by
Levenshtein~\cite{levenshtein2002efficient}, and has since been developed in
related models of information retrieval with multiple clues and list
reconstruction~\cite{junnila2015information,junnila2023levenshtein}. In the
substitution-error setting, the ambiguity of reconstruction is naturally
controlled by the maximum possible size of an intersection of Hamming balls.
This connection was exploited by Yaakobi and Bruck~\cite{yaakobi2018uncertainty}
in the binary setting, and by Junnila et al. in several fixed-parameter
and $q$-ary variants~\cite{junnila2020levenshtein,junnila2023levenshtein,junnila2025intersections}. Besides, related reconstruction and list-decoding problems in the Levenshtein
metric, including insertions and deletions, were studied by Abu-sini and Yaakobi~\cite{abu2021list,abu2021levenshtein,abu2023intersection}.

The case $m=2$ has been extensively studied, partly because of its role in
sequence reconstruction and in Gilbert--Varshamov type bounds. Levenshtein
initiated the study of intersections of two Hamming balls~\cite{levenshtein2002efficient}. Subsequent work analyzed their asymptotic
behavior when the radius and the distance between centers grow linearly with
the dimension. Jiang and Vardy~\cite{jiang2004asymptotic} studied the binary
case using computer-assisted methods, while Vu and Wu~\cite{vu2005improving}
extended related techniques to the $q$-ary setting. More recently, Kim
et al.~\cite{kim2022exponential} showed that the ratio between the intersection
volume and the volume of a single ball decays exponentially with the distance
between the centers.

For three or more balls, the situation is more delicate. In the binary case,
Yaakobi and Bruck~\cite{yaakobi2018uncertainty} derived an exact combinatorial
formula for the intersection of three Hamming balls and showed that
$N_r^{(n)}(m,k)=\Theta(n^{r-\lceil k/2\rceil})$ when $r,k,m$ are fixed constants
and $n$ is sufficiently large. Junnila and coauthors further studied related
fixed-parameter problems for substitution errors, including the cases of larger
alphabets, different error patterns, and exact formulas for
$N_r^{(n)}(m,k)$ under suitable assumptions on the parameters and sufficiently
large $n$~\cite{junnila2020levenshtein,junnila2023levenshtein,junnila2025intersections}.

In this paper, we study the binary three-ball problem in the regime $r=\alpha n, k=\beta n$, where $\alpha,\beta\in(0,1)$ are fixed constants. We treat $\alpha$ as the
relative radius and $\beta$ as the relative pairwise center distance. Our goal
is to determine the normalized logarithmic growth of $N_r^{(n)}(3,k)$, namely the
exponential rate of the largest three-ball intersection.

Our approach is based on a sphere decomposition. Instead of estimating the
intersection of three balls directly, we decompose each ball into Hamming
spheres and show that, up to polynomial factors, the exponential rate of the
ball intersection is determined by the largest intersection of three spheres.
This reduces the problem to a type-counting calculation.  The main result of this paper is an entropy variational formula for the
exponential rate of the largest three-ball intersection. In the linear regime with
$r=\alpha n$ and $k=\beta n$, we prove that $
N_{\alpha n}^{(n)}(3,\beta n)
=\exp\left((f_3(\alpha,\beta)+o(1))n\right)$, where the rate function $f_3(\alpha,\beta)$ is characterized by a
two-dimensional constrained entropy maximization problem. This variational
formula is derived by decomposing balls into spheres, reducing the resulting
sphere intersections to block-type counts, and applying Stirling's formula.

We also compare the three-ball exponent with the corresponding two-ball
exponent. Let $g_2(\alpha,\beta)$ denote the exponential rate of the
intersection of two Hamming balls of radius $\alpha n$ whose centers have
distance $\beta n$. Numerical evaluations show that the three-ball exponent
is smaller than the two-ball exponent, i.e., $f_3(\alpha,\beta)<g_2(\alpha,\beta)$. Moreover, the gap becomes especially pronounced near the critical boundary
$\beta=2\alpha$. As $\beta\to 2\alpha$, the three-ball exponent degenerates to 0.
At the boundary itself, the variational formula gives $f_3(\alpha,2\alpha)=0$.
Thus the three-ball intersection has no positive exponential growth rate at
the critical boundary. We further refine this statement in the critical window
$k=2t$ and $r=t+C$, where $C$ is fixed, by showing that the intersection is
polynomially bounded. This demonstrates a transition from exponential-size
intersections in the interior regime $\beta<2\alpha$ to polynomial-size
intersections at the boundary.

\section{Problem Setup and Extremal Reduction}

We study the maximum intersection of three Hamming balls whose centers are
separated by distance at least $k$:
\begin{equation*}
N_r^{(n)}(3,k)
=
\max_{\substack{x_1,x_2,x_3\in\Sigma^n\\ d(x_i,x_j)\ge k,\ i\ne j}}
\left|
B_r(x_1)\cap B_r(x_2)\cap B_r(x_3)
\right|.
\end{equation*}
Throughout the paper, we consider the binary case, i.e., $q=2$ and focus on the linear regime $r=\alpha n, k=\beta n$, where $\alpha,\beta\in(0,1)$ are fixed. While the three-ball intersection has
an exact combinatorial formula in prior work~\cite{yaakobi2018uncertainty}, that formula does not directly
yield a closed-form expression for the exponential rate in this linear-scaling
regime. We therefore study the normalized logarithmic growth of the
intersection, treating $\alpha$ and $\beta$ as the relative radius and relative
center distance. 

We first record the nontrivial parameter regime. If $\alpha<\beta/2$, then
$k>2r$, and the triangle inequality implies that even two balls with centers
at distance at least $k$ are disjoint. Hence $N_r^{(n)}(3,k)=0$. Therefore we
assume $\alpha\ge\beta/2$.There is also a feasibility constraint on the centers. For any three binary
vectors $x_1,x_2,x_3$, each coordinate contributes at most $2$ to $d(x_1,x_2)+d(x_1,x_3)+d(x_2,x_3)$. Consequently, $
d(x_1,x_2)+d(x_1,x_3)+d(x_2,x_3)\le 2n$. If all three pairwise distances are at least $k$, then $3k\le 2n$, or
equivalently $\beta\le 2/3$. Thus the nontrivial regime is $\beta/2\le \alpha,0<\beta\le 2/3$.

We next reduce the maximization over centers to a canonical configuration.
In \cite{cohen1997covering} (and \cite{honkala1987intersection}), it has been shown that if $q=2$ the intersection size $N_r^{(n)}(3,k)$
depends only on the three pairwise distances among the centers. A monotonicity
argument shows that decreasing these distances, while keeping the constraint
$d(x_i,x_j)\ge k$ feasible, cannot decrease the common intersection. Hence the
maximum is attained when the pairwise distances are minimal.

\begin{lemma}[Extremal center configuration]
For even $k=2t$, the maximum defining $N_r^{(n)}(3,k)$ is attained by a
triple of centers satisfying
\begin{equation*}
d(x_1,x_2)=d(x_1,x_3)=d(x_2,x_3)=k.
\end{equation*}
For odd $k$, the maximum is attained by a triple whose pairwise distances
are $k,k,k+1$, up to permutation.
\end{lemma}

For the remainder of the main analysis, we treat the even case $k=2t$.
The odd case is analogous and differs only in rounding terms. By symmetry of the
Hamming cube, we choose the centers as
\begin{equation*}
a=0^n,\qquad
b=1^{2t}0^{n-2t},\qquad
c=1^t0^t1^t0^{n-3t}.
\end{equation*}
These centers satisfy $d(a,b)=d(a,c)=d(b,c)=2t=k$. Define the common intersection set and its cardinality by
\begin{equation*}
\mathcal I_r(a,b,c)
=
B_r(a)\cap B_r(b)\cap B_r(c),\, I_r(a,b,c)
=
|\mathcal I_r(a,b,c)|.
\end{equation*}
Therefore, for even $k$, the original extremal quantity reduces to $
N_r^{(n)}(3,k)=I_r(a,b,c)$.

\section{Sphere Reduction and Counting Three-Sphere Intersections}

\subsection{Decomposing balls into spheres}

For $x\in\Sigma^n$ and $0\le s\le n$, define the Hamming sphere as $
V_s(x)=\{y\in\Sigma^n:d(x,y)=s\}$. Then, we have
\begin{equation*}
B_r(x)=\bigcup_{s=0}^{r} V_s(x),
\end{equation*}
where the union is disjoint. Therefore, for the canonical centers $a,b,c$
fixed above, we have
\begin{equation*}
\mathcal I_r(a,b,c)
=
\bigcup_{0\le r_1,r_2,r_3\le r}
\left(
V_{r_1}(a)\cap V_{r_2}(b)\cap V_{r_3}(c)
\right).
\end{equation*}
Consequently, we also have the size of $\mathcal I_r(a,b,c)$
\begin{equation*}
I_r(a,b,c)=
\sum_{0\le r_1,r_2,r_3\le r}
|V_{r_1}(a)\cap V_{r_2}(b)\cap V_{r_3}(c)|.
\end{equation*}

Then, we define the largest contribution from a single
distance profile $(r_1,r_2,r_3)$ as
\begin{equation*}
M_r(a,b,c)
=
\max_{0\le r_1,r_2,r_3\le r}
|V_{r_1}(a)\cap V_{r_2}(b)\cap V_{r_3}(c)|.
\end{equation*}
Since the above sum contains at most $(r+1)^3$ terms, we have
\begin{equation*}
M_r(a,b,c)
\le I_r(a,b,c)\le (r+1)^3 M_r(a,b,c).
\end{equation*}
Because $r=\alpha n=O(n)$, the prefactor $(r+1)^3$ is polynomial in $n$ and
therefore does not affect the exponential rate, i.e., $1/n\cdot\log (r+1)^3=o(1)$.
Thus, whenever the intersection is nonempty, we have
\begin{equation*}
\frac{1}{n}\log I_r(a,b,c)
=
\frac{1}{n}\log M_r(a,b,c)+o(1).
\end{equation*}
Hence, to determine the asymptotic exponential rate of the three-ball
intersection, it suffices to estimate the largest intersection of three
Hamming spheres with radii at most $r$.

\subsection{Canonical centers}

We now compute $|V_{r_1}(a)\cap V_{r_2}(b)\cap V_{r_3}(c)|$ for the canonical centers $a=0^n, b=1^{2t}0^{n-2t}, c=1^t0^t1^t0^{n-3t}$, where $k=2t$. These centers induce a partition of the coordinates into four
blocks of sizes $t,t,t,n-3t$. In these four blocks, the three centers have the
following forms:
\begin{equation*}
\begin{array}{c|cccc}
 & I_1 & I_2 & I_3 & I_4 \\
\hline
a & 0^t & 0^t & 0^t & 0^{n-3t} \\
b & 1^t & 1^t & 0^t & 0^{n-3t} \\
c & 1^t & 0^t & 1^t & 0^{n-3t}.
\end{array}
\end{equation*}

Fix radii $r_1,r_2,r_3\le r$, and let $x\in V_{r_1}(a)\cap V_{r_2}(b)\cap V_{r_3}(c)$. Let $a_j$ denote the number of ones of $x$ in block $I_j$, for
$j=1,2,3,4$. Then $x$ has the block form
\begin{equation*}
x=
1^{a_1}0^{t-a_1}
1^{a_2}0^{t-a_2}
1^{a_3}0^{t-a_3}
1^{a_4}0^{n-3t-a_4}.
\end{equation*}
The three distance constraints $d(x,a)=r_1$, $d(x,b)=r_2$, and
$d(x,c)=r_3$ give
\begin{align*}
&a_1+a_2+a_3+a_4=r_1,\\
&(t-a_1)+(t-a_2)+a_3+a_4=r_2,\\
&(t-a_1)+a_2+(t-a_3)+a_4=r_3.
\end{align*}
Thus, once $a_1$ is fixed, the remaining variables are determined as follows
\begin{align}\label{eq:a234}
&a_2=(r_1+2t-r_2-2a_1)/2,\nonumber\\
&a_3=(r_1+2t-r_3-2a_1)/2,\nonumber\\
&a_4=(2a_1+r_2+r_3-4t)/2.
\end{align}

For fixed $a_1,r_1,r_2,r_3$, let $S_{a_1,r_1,r_2,r_3}$ denote the number of vectors $x$ in the three-sphere intersection with exactly $a_1$ ones in the first coordinate block:
\begin{equation*}
S_{a_1,r_1,r_2,r_3}
=
\binom{t}{a_1}
\binom{t}{a_2}
\binom{t}{a_3}
\binom{n-3t}{a_4},
\end{equation*}
where $a_2, a_3, a_4$ can be written as Eqn.~\eqref{eq:a234} in terms of $a_1, r_1, r_2$ and $r_3$. The first binomial chooses the positions of the ones of $x$ in block $I_1$; the second and third binomials choose the positions of the ones in blocks
$I_2$ and $I_3$; and the last binomial chooses the positions of the ones in
block $I_4$.

Therefore, for fixed $r_1,r_2,r_3$, the size of the three-sphere intersection
is
\begin{equation*}
|V_{r_1}(a)\cap V_{r_2}(b)\cap V_{r_3}(c)|
=
\sum_{a_1=0}^{t} S_{a_1,r_1,r_2,r_3}.
\end{equation*}
Since the sum over $a_1$ has at most $t+1$ terms, it is exponentially
equivalent to its largest term:
\begin{multline*}
\max_{0\le a_1\le t} S_{a_1,r_1,r_2,r_3}
\le
|V_{r_1}(a)\cap V_{r_2}(b)\cap V_{r_3}(c)|
\\\le
(t+1)\max_{0\le a_1\le t} S_{a_1,r_1,r_2,r_3}.
\end{multline*}
Combining this bound with the sphere reduction gives us
\begin{equation*}
\frac{1}{n}\log I_r(a,b,c)
=
\max_{\substack{0\le r_1,r_2,r_3\le r\\ 0\le a_1\le t\\
S_{a_1,r_1,r_2,r_3}>0}}
\frac{1}{n}\log S_{a_1,r_1,r_2,r_3}
+o(1).
\end{equation*}
Thus the original three-ball problem has been reduced to a finite-dimensional
discrete optimization over the parameters $a_1,r_1,r_2,r_3$. In the next
section, we pass to normalized variables and use Stirling's formula to obtain
the corresponding entropy maximization problem.

\section{Entropy Rate for Three Balls}

In the previous section, we reduced the three-ball intersection problem to the
maximization of $S_{a_1,r_1,r_2,r_3}$. We now simplify this discrete
optimization and pass to normalized variables. The goal is to identify the
exponential rate of $I_r(a,b,c)$ as a variational problem depending only on
the relative radius $\alpha$ and the relative center distance $\beta$.

\subsection{Symmetry reduction: $r_2=r_3$}

The centers $b$ and $c$ play symmetric roles in the canonical construction.
Therefore, one expects the dominant distance profile to have equal distances
from $x$ to $b$ and $c$. The next lemma formalizes this intuition.

\begin{lemma}[Symmetry reduction]
For determining the exponential rate of $I_r(a,b,c)$, it suffices to consider
terms with $r_2=r_3$. Equivalently, the maximum of
$S_{a_1,r_1,r_2,r_3}$ can be attained at a distance profile satisfying
\begin{equation*}
r_2=r_3.
\end{equation*}
\end{lemma}

\begin{proof}
Fix $a_1$, $r_1$, and the sum $r_2+r_3$. In the expression
$S_{a_1,r_1,r_2,r_3}$, the first binomial coefficient depends only on $a_1$,
and the last binomial coefficient depends only on $r_2+r_3$. Thus, for fixed
$a_1$, $r_1$, and $r_2+r_3$, the only terms affected by changing $r_2$ and
$r_3$ are the two middle binomial coefficients.

We define
\begin{equation*}
u_2=\frac{r_1+2t-r_2-2a_1}{2},
\qquad
u_3=\frac{r_1+2t-r_3-2a_1}{2}.
\end{equation*}
Here $u_2$ and $u_3$ are the numbers
of ones of $x$ in the second and third coordinate blocks, respectively. Since
$r_2+r_3$ is fixed, the sum $u_2+u_3$ is also fixed. Hence, we need to maximize $\binom{t}{u_2}\binom{t}{u_3}$ over $u_2+u_3$ fixed. By the unimodality of binomial coefficients, this
product is maximized when $u_2$ and $u_3$ are as balanced as possible. Up to
integer rounding, this gives $u_2=u_3$. By the definitions of $u_2$ and $u_3$, this is equivalent to $r_2=r_3$. Thus, the dominant contribution may be taken from distance profiles with $r_2=r_3$.
\end{proof}

After this reduction, we write $r_2=r_3$ and denote this common value by
$r_2$. The variable $r_2$ now represents the common distance from the candidate
point $x$ to the two symmetric centers $b$ and $c$. The corresponding term is
\begin{equation*}
S_{a_1,r_1,r_2,r_2}
=
\binom{t}{a_1}
\binom{t}{\frac{r_1+2t-r_2-2a_1}{2}}^2
\binom{n-3t}{a_1+r_2-2t}.
\end{equation*}

\subsection{Boundary reduction for $r_1$}

We next optimize over $r_1$ for fixed $a_1$ and $r_2$. The variable $r_1$ is
the distance from $x$ to the center $a=0^n$. For fixed $a_1$ and $r_2$, the
first and last binomial coefficients in $S_{a_1,r_1,r_2,r_2}$ do not depend on
$r_1$. Only the middle binomial coefficient depends on $r_1$.

Let $q=(r_1+2t-r_2-2a_1)/2$. The variable $q$ is the number of ones of $x$ in each of the second and third
coordinate blocks after the symmetry reduction. For fixed $a_1$ and $r_2$,
maximizing $S_{a_1,r_1,r_2,r_2}$ over $r_1$ is therefore equivalent to making
$\binom{t}{q}$ as large as possible.

Since $\binom{t}{q}$ is maximized when $q$ is as close as possible to $t/2$,
the unconstrained optimal choice satisfies
\begin{equation*}
r_1+2t-r_2-2a_1
=2.
\end{equation*}
Equivalently, we have $r_1=r_2+2a_1-t$. However, $r_1$ is a sphere radius inside the ball $B_r(a)$, so it must satisfy
$r_1\le r$. Hence the maximizing choice is
\begin{equation*}
r_1^\star(a_1,r_2)
=
\min\{r,\ r_2+2a_1-t\},
\end{equation*}
whenever this value is feasible.

This reduction has a simple interpretation. For fixed $a_1$ and common
distance $r_2$, the best choice of $r_1$ tries to balance the number of ones
placed in the second and third coordinate blocks. If this balancing choice
exceeds the allowed radius $r$, the optimum is forced to lie on the boundary
$r_1=r$.

\subsection{Main theorem: entropy maximization formula}

We now pass to normalized variables. After the symmetry reduction
$r_2=r_3$ and the optimization over $r_1$, the remaining free discrete
variables are $a_1$ and the common radius $r_2$. We define
\begin{equation*}
\theta=a_1/n,
\qquad
\delta=r_2/n.
\end{equation*}
Here $\theta$ is the normalized number of ones of $x$ in the first coordinate
block, and $\delta$ is the normalized common distance from $x$ to the two
symmetric centers $b$ and $c$.

For fixed $a_1$ and $r_2$, the optimal value of $r_1$ is
\begin{equation*}
r_1^\star(a_1,r_2)
=
\min\{r,\ r_2+2a_1-t\}.
\end{equation*}
We denote its normalized value by
\begin{equation*}
\rho^\star(\theta,\delta)
=
r_1^\star(a_1,r_2)/n.
\end{equation*}
Thus $\rho^\star(\theta,\delta)$ is the optimized normalized distance from
$x$ to the center $a$. Since $r=\alpha n$ and $t=\beta n/2$, we have
\begin{equation*}
\rho^\star(\theta,\delta)
=
\min\left\{
\alpha,\ \delta+2\theta-\beta/2
\right\}.
\end{equation*}
Importantly, $\rho^\star(\theta,\delta)$ is not an additional free variable;
it is determined by the two free variables $\theta$ and $\delta$.

We next define the block densities induced by $\theta$, $\delta$, and
$\rho^\star(\theta,\delta)$. First, define $p_1(\theta)
=
\frac{2\theta}{\beta}$, where $p_1(\theta)$ is the fraction of ones of $x$ in the first block,
whose length is $t=\beta n/2$.

Next, we define
\begin{equation*}
p_{23}(\theta,\delta)
=(\rho^\star(\theta,\delta)+\beta-\delta-2\theta)/\beta.
\end{equation*}
The quantity $p_{23}(\theta,\delta)$ is the common fraction of ones of $x$ in
the second and third blocks after imposing the symmetry reduction $r_2=r_3$.
These two blocks have the same length $t=\beta n/2$ and the same number of
ones in the dominant profile.

Finally, we define
\begin{equation*}
p_4(\theta,\delta)
=
(\theta+\delta-\beta)/(1-3\beta/2).
\end{equation*}
The quantity $p_4(\theta,\delta)$ is the fraction of ones of $x$ in the fourth
block, whose length is $n-3t=(1-3\beta/2)n$.

Let $H$ be the binary entropy function with natural logarithms as
$H(p)=-p\log p-(1-p)\log(1-p), p\in[0,1]$.
The value $H(p)$ gives the exponential rate of choosing a fraction $p$ of
positions from a large coordinate block.

Then, we can define the feasible set as
\begin{multline*}
\mathcal D_{\alpha,\beta}
=
\Big\{
(\theta,\delta):
0\le \delta\le \alpha,\
0\le \rho^\star(\theta,\delta)\le \alpha,\ \\
0\le p_1(\theta)\le 1,\
0\le p_{23}(\theta,\delta)\le 1,\
0\le p_4(\theta,\delta)\le 1
\Big\}.
\end{multline*}
The set $\mathcal D_{\alpha,\beta}$ is the continuous version of the discrete
feasibility constraints. Equivalently, it contains exactly those normalized
profiles for which all four binomial coefficients in the counting formula have
valid arguments.

For $(\theta,\delta)\in\mathcal D_{\alpha,\beta}$, we define
\begin{multline*}
\Phi_{\alpha,\beta}(\theta,\delta)
=
\beta/2\cdot H(p_1(\theta))
+
\beta H(p_{23}(\theta,\delta))
\\+
\left(1-3\beta/2\right)H(p_4(\theta,\delta)).
\end{multline*}
The function $\Phi_{\alpha,\beta}(\theta,\delta)$ is the normalized logarithmic
size of the block-composition class determined by $(\theta,\delta)$. The first
term comes from the first block, the second term combines the symmetric second
and third blocks, and the third term comes from the fourth block.

We then define
\begin{equation*}
f_3(\alpha,\beta)
=
\max_{(\theta,\delta)\in\mathcal D_{\alpha,\beta}}
\Phi_{\alpha,\beta}(\theta,\delta).
\end{equation*}
The quantity $f_3(\alpha,\beta)$ is the asymptotic exponential rate of the
three-ball intersection.

\begin{theorem}[Entropy rate of three-ball intersection]
Assume $0<\beta<2/3$ and $\alpha\ge \beta/2$. In the even-distance case
$k=2t$, with $r=\alpha n$ and $k=\beta n$, the canonical three-ball
intersection satisfies
\begin{equation*}
\frac{1}{n}\log I_r(a,b,c)
=
f_3(\alpha,\beta)+o(1).
\end{equation*}
Equivalently, $I_r(a,b,c)=\exp\left((f_3(\alpha,\beta)+o(1))n\right)$. By the extremal reduction, the same exponential rate gives
\begin{equation*}
N_r^{(n)}(3,k)
=
\exp\left((f_3(\alpha,\beta)+o(1))n\right).
\end{equation*}
\end{theorem}

\begin{proof}
From the sphere reduction and the binomial formula, we have
\begin{equation*}
\frac{1}{n}\log I_r(a,b,c)
=
\max_{\substack{0\le r_1,r_2,r_3\le r\\ 0\le a_1\le t\\
S_{a_1,r_1,r_2,r_3}>0}}
\frac{1}{n}\log S_{a_1,r_1,r_2,r_3}
+o(1).
\end{equation*}
The symmetry reduction allows us to restrict to $r_2=r_3$, up to rounding
terms that contribute only $o(1)$ to the normalized logarithm. The boundary
reduction then optimizes over $r_1$ and gives $r_1=r_1^\star(a_1,r_2)$.
Thus the remaining maximization is over $a_1$ and $r_2$, or equivalently over
their normalized versions $\theta$ and $\delta$.

For feasible $(\theta,\delta)$, the four binomial coefficients in
$S_{a_1,r_1^\star,r_2,r_2}$ correspond to block lengths $t, t, t, n-3t,$
with success fractions
$p_1(\theta), p_{23}(\theta,\delta), p_{23}(\theta,\delta), p_4(\theta,\delta)$,
respectively. By Stirling's formula, uniformly over feasible profiles,
\begin{multline*}
1/n\cdot\log S_{a_1,r_1^\star,r_2,r_2}
=\beta/2\cdot H(p_1(\theta))+
\beta H(p_{23}(\theta,\delta))
\\+
\left(1-3\beta/2\right)H(p_4(\theta,\delta))
+
o(1).
\end{multline*}
Therefore, we can obtain that
\begin{equation*}
1/n\cdot\log S_{a_1,r_1^\star,r_2,r_2}
=
\Phi_{\alpha,\beta}(\theta,\delta)+o(1).
\end{equation*}
Maximizing over the discrete feasible lattice is equivalent, up to $o(1)$, to
maximizing $\Phi_{\alpha,\beta}$ over the continuous feasible set
$\mathcal D_{\alpha,\beta}$. Hence
\begin{multline*}
1/n\cdot\log I_r(a,b,c)
=
\max_{(\theta,\delta)\in\mathcal D_{\alpha,\beta}}
\Phi_{\alpha,\beta}(\theta,\delta)
+
o(1)
\\=
f_3(\alpha,\beta)+o(1).
\end{multline*}
This proves the theorem.
\end{proof}

\section{Numerical Evaluation and Critical Behavior}

The previous section characterizes the three-ball intersection through the
variational formula
\begin{equation*}
f_3(\alpha,\beta)
=
\max_{(\theta,\delta)\in\mathcal D_{\alpha,\beta}}
\Phi_{\alpha,\beta}(\theta,\delta).
\end{equation*}
Here $f_3(\alpha,\beta)$ is the exponential rate of the three-ball
intersection, $\theta$ represents the normalized number of ones in the first
coordinate block, and $\delta$ represents the normalized common distance from
the candidate vector to the two symmetric centers $b$ and $c$. The feasible set
$\mathcal D_{\alpha,\beta}$ enforces that all block densities are valid
probabilities and that all sphere radii remain within the ball radius.

Although the formula is explicit, it does not generally simplify to a single
closed-form expression. The main reason is that the optimized normalized
distance to the first center, $\rho^\star(\theta,\delta)
=\min\left\{
\alpha,\ \delta+2\theta-\beta/2\right\}$, is piecewise linear. Thus the maximizer may lie in different active-constraint
regions depending on the values of $\alpha$ and $\beta$. Equivalently, the
optimizer may either occur in the region where the balancing choice of $r_1$
is feasible, or on the boundary where $r_1=r$. For this reason, a fully
expanded expression for $f_3(\alpha,\beta)$ would require a lengthy
case-by-case description of the active constraints.

In our numerical evaluation, we compute $f_3(\alpha,\beta)$ directly from the
two-dimensional variational formula. Concretely, for each fixed pair
$(\alpha,\beta)$, we optimize $\Phi_{\alpha,\beta}(\theta,\delta)$ over the
feasible set $\mathcal D_{\alpha,\beta}$. Since the optimization is only
two-dimensional, this can be done reliably by a dense grid search followed by a
local constrained refinement. Let $\mathcal G_N$ denote a grid of mesh size
$1/N$ in the $(\theta,\delta)$ plane. The grid approximation is $f_{3,N}(\alpha,\beta)
=
\max_{(\theta,\delta)\in\mathcal D_{\alpha,\beta}\cap\mathcal G_N}
\Phi_{\alpha,\beta}(\theta,\delta)$.
The quantity $f_{3,N}(\alpha,\beta)$ is the best exponent found among the
finite collection of normalized block-composition profiles in the grid.

We compare this three-ball exponent with the corresponding exponent for the
intersection of two Hamming balls. Let $g_2(\alpha,\beta)$ denote the
two-ball exponent, where the two balls have radius $\alpha n$ and their
centers have distance $\beta n$. In the regime $\beta/2\le \alpha\le 1/2$, the
standard sphere-decomposition calculation gives
\begin{equation*}
g_2(\alpha,\beta)
=
\beta\log 2
+
(1-\beta)
H\left(
(\alpha-\beta/2)/(1-\beta)
\right).
\end{equation*}
Here the term $\beta\log 2$ comes from the coordinates on which the two
centers differ, while the entropy term comes from the remaining coordinates.

Figure~\ref{fig:rate-comparison} compares $f_3(\alpha,\beta)$ and
$g_2(\alpha,\beta)$ for several fixed values of $\alpha$ while varying
$\beta$ toward the critical value $2\alpha$.

\begin{figure}[t]
    \centering
    \includegraphics[width=0.32\textwidth]{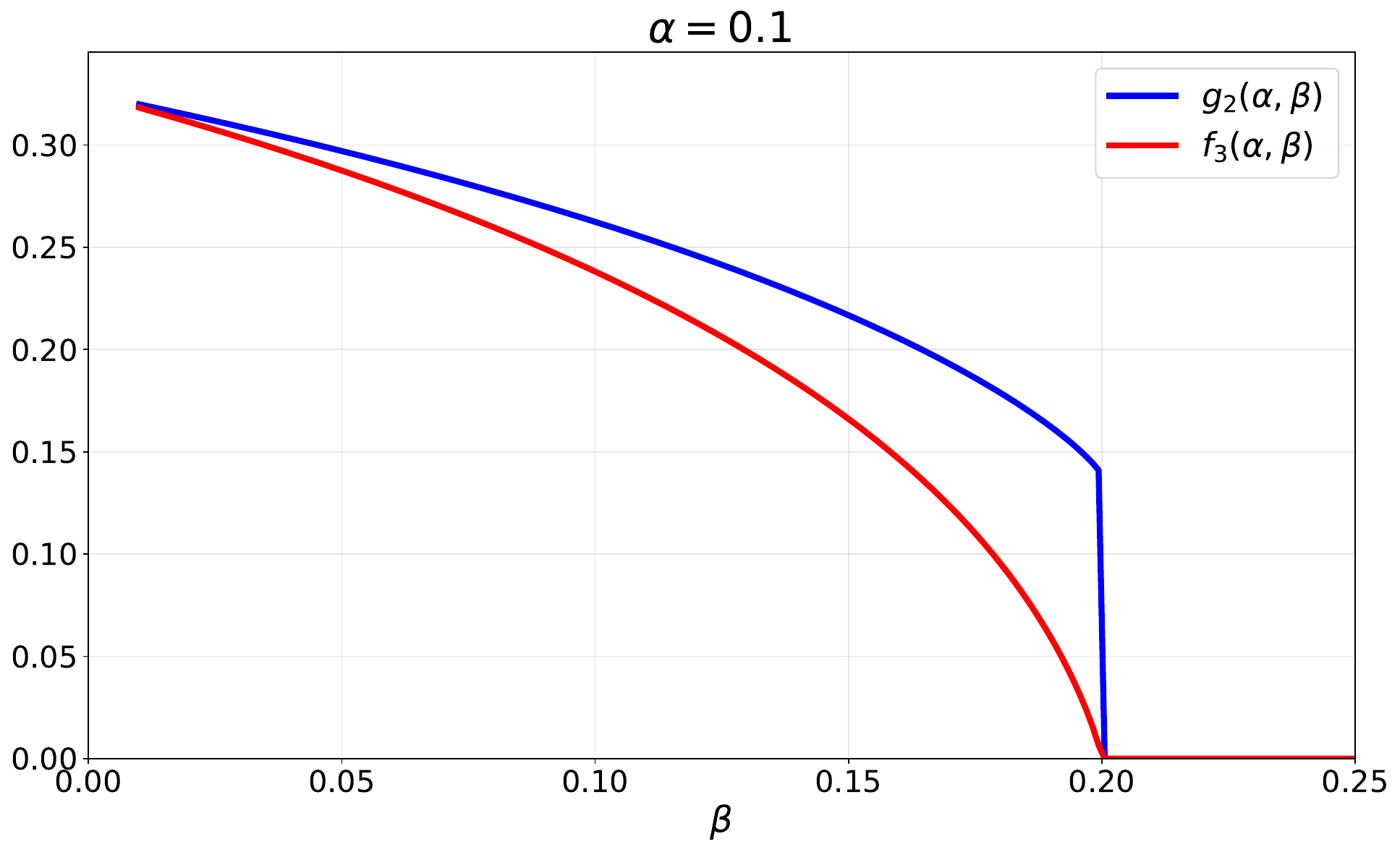}
    \includegraphics[width=0.32\textwidth]{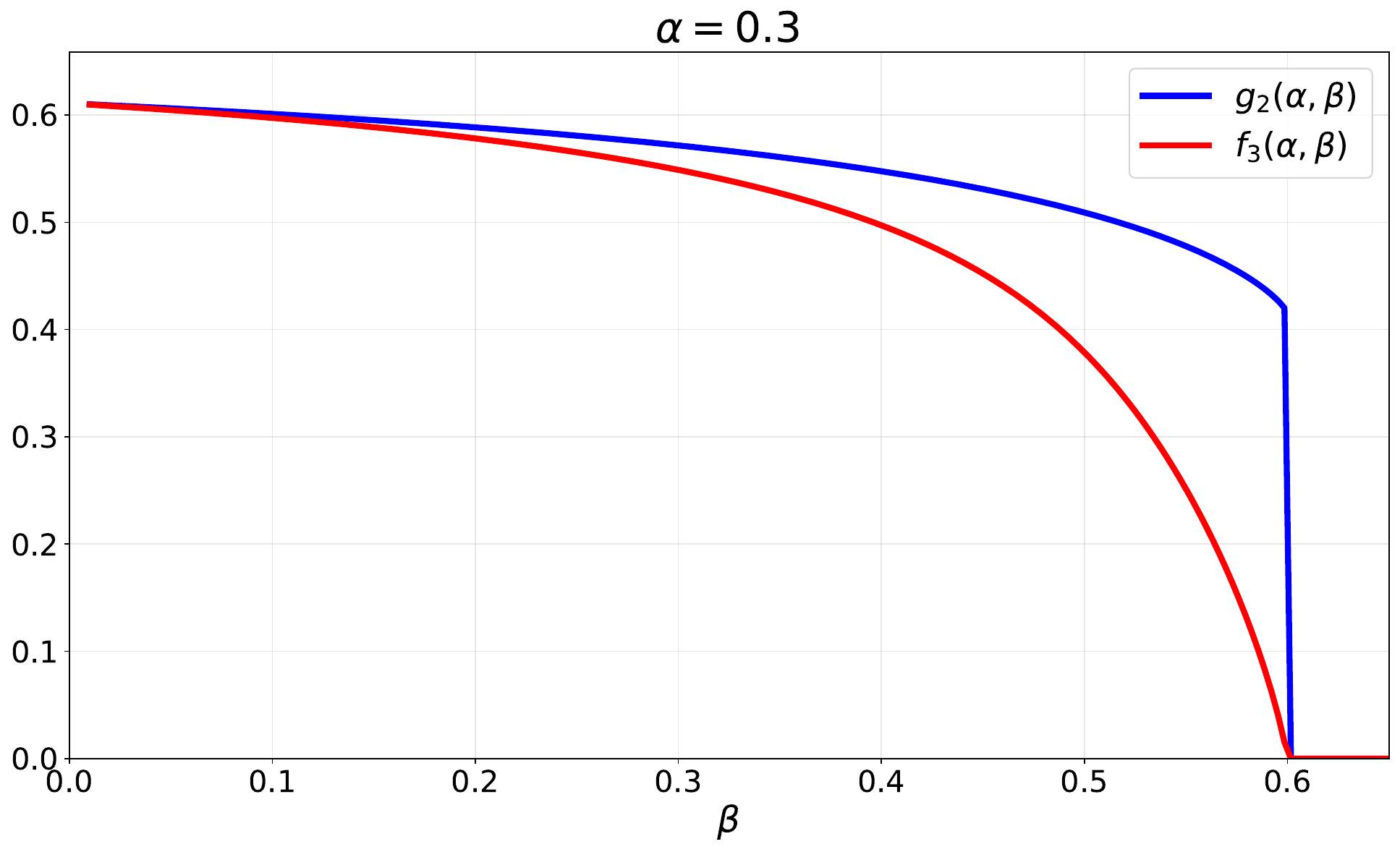}
    \includegraphics[width=0.32\textwidth]{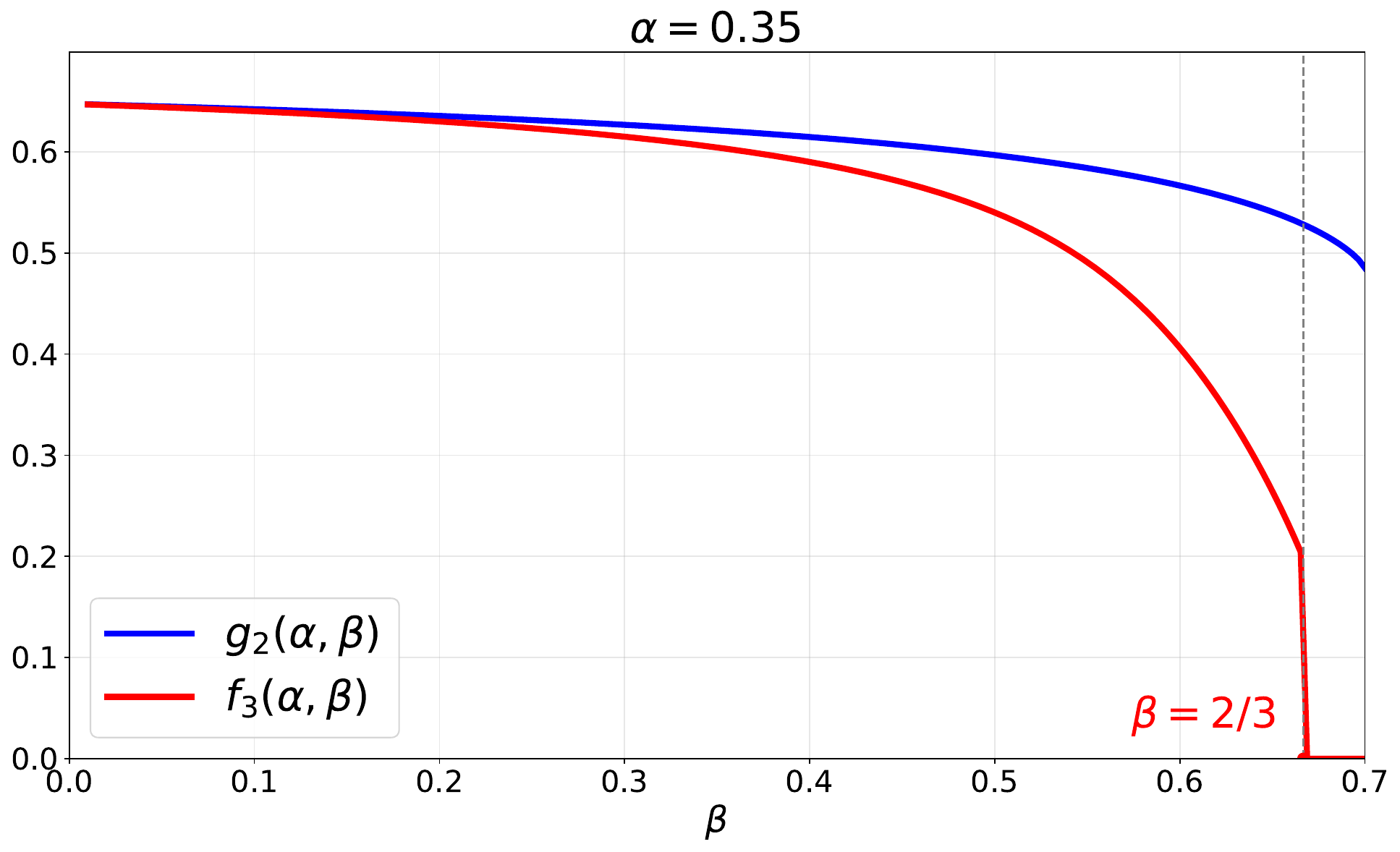}
    \caption{
    Comparison between the three-ball exponent $f_3(\alpha,\beta)$ and the
    two-ball exponent $g_2(\alpha,\beta)$ for several fixed values of
    $\alpha$. The three-ball exponent is consistently smaller, and it
    vanishes as $\beta$ approaches $2\alpha$.
    }
    \label{fig:rate-comparison}
\end{figure}

The numerical results show two clear phenomena. First, the three-ball
intersection has a smaller exponential rate than the two-ball intersection, i.e.,
$f_3(\alpha,\beta)\le g_2(\alpha,\beta)$. This is expected, since adding a third ball
imposes an additional consistency constraint on the candidate vectors. Second,
the difference becomes especially pronounced near the critical boundary
$\beta=2\alpha$. As $\beta\rightarrow 2\alpha$, the three-ball exponent
degenerates to 0.
In contrast, the two-ball exponent remains positive at the same boundary.
Thus the third ball does not merely reduce the intersection by a subexponential
or constant factor; near the boundary, it collapses the exponential rate.

We now explain the critical boundary more explicitly. The condition
$\beta=2\alpha$ is equivalent to $k=2r$, meaning that the distance between
centers is twice the ball radius. This is the largest separation at which two
radius-$r$ balls can still touch. In the three-ball case, the feasible entropy
profile becomes degenerate. Indeed, when $\beta=2\alpha$, feasibility forces
the normalized profile to satisfy $
\theta=\beta/2,
\delta=\beta/2$.
At this profile, the corresponding block densities are $p_1=1, p_{23}=0, p_4=0.$ Then,
$f_3(\alpha,2\alpha)=0$.

The identity $f_3(\alpha,2\alpha)=0$ shows that the three-ball intersection
has no positive exponential growth rate at the critical boundary. However, it
does not by itself determine the remaining subexponential size. To refine this
statement, consider the critical window as $k=2t,
r=t+C$, where $C$ is a fixed integer independent of $n$. The parameter $C$ measures the
fixed excess radius beyond the critical value $t=k/2$. In this regime, the
intersection is polynomially bounded as
\begin{equation*}
I_{t+C}(a,b,c)
=
O(n^{3C}).
\end{equation*}
This polynomial
bound is consistent with the vanishing exponent and shows that the boundary
$\beta=2\alpha$ marks a sharp transition: in the interior regime
$\beta<2\alpha$, the three-ball intersection may have exponential size, while
at the critical boundary its size is at most polynomial.
\newpage
\bibliographystyle{IEEEtran}
\bibliography{reference}
\end{document}